\documentclass[11pt]{amsart}
\usepackage{amscd,amssymb,palatino,amsmath}
\usepackage[utf8]{inputenc}
\usepackage[english]{babel}
\usepackage{tikz-cd}
\usepackage{amssymb}
\usepackage{graphicx, color}
\usepackage[all,cmtip]{xy}
\usepackage{pb-diagram,pb-xy} 

\newtheorem{proposition}{Proposition}
\newtheorem{theorem}[proposition]{Theorem}
\newtheorem{definition}[proposition]{Definition}
\newtheorem{lemma}[proposition]{Lemma}

\newtheorem{corollary}[proposition]{Corollary}
\newtheorem{remark}[proposition]{Remark}

\title{Non-existence of Global Transverse Poincar\'{e} Sections}
\author{Roisin Braddell}\address{ Roisin Braddell,
Laboratory of Geometry and Dynamical Systems, Department of Mathematics, Universitat Polit\`{e}cnica de Catalunya and BGSMath, Barcelona  \it{e-mail: roisin.braddell@upc.edu }
}
\thanks{The author is supported by a partial predoctoral grant from UPC with the ICREA Academia project of Eva Miranda. The author is partially supported  by the grants reference number MTM2015-69135-P (MINECO/FEDER) and reference number 2017SGR932 (AGAUR). This material is based upon work supported by the National Science Foundation under Grant No. DMS-1440140 and was begun while the author was in residence at the Mathematical Sciences Research Institute in Berkeley, California, during the Fall 2018 semester.}

\date{April 2019}

\begin{document}
\maketitle
\begin{abstract}
We study global transverse Poincar\'{e} sections and give topological conditions for their existence, showing they never exist in many important cases. We prove that an energy hypersurface possessing global transverse Poincar\'{e} section is equivalent to the hypersuface having a cosymplectic structure. We give a family of Hamiltonian systems with global Poincar\'{e} sections of all possible topologies. Finally, we address the question of when a compact hypersurface of a symplectic manifold possesses an induced cosymplectic structure.
\end{abstract}

\section{Introduction}
There has been a recent resurgence of interest in \emph{global surfaces of section}, whose applications range from systolic inequalities in Riemannian geometry \cite{abbondandolo2017systolic} to the existence of periodic orbits in equations of celestial mechanics (see for instance \cite{conley1963some} and \cite{albers2012global} for two examples among many). The difficulty of finding closed global surfaces of section (sometimes called \emph{ideal sections} or alternatively \emph{global transverse Poincar\'{e} sections}) was already noted by Birkhoff \cite{birkhoff1917dynamical}. As a result, one is usually considers global sections with boundary (as in the case of \cite{conley1963some} or \cite{albers2012global}) or tries to study sections which are transverse but not complete in the traditional sense (as found in \cite{dullin1995complete}).\\
\indent Necessary topological conditions for the existence of ideal sections have been given in, for example, \cite{bolsinov1996topology} where it was proved that that the existence of such a section is impossible for compact energy surfaces of natural two degree of freedom Hamiltonian systems of most topologies. Other work includes \cite{dullin1995complete}, which proves that ideal sections do not exist for energy surfaces of time reversal symmetric Hamiltonians. A good general description of the conditions for global transverse  Poincar\'{e} sections, however, is still lacking. \\
\indent Here, we prove a compact energy surface in possessing such a section has a \emph{cosymplectic} structure. This immediately gives topological restrictions on energy surfaces possessing such sections, generalizing results found in \cite{bolsinov1996topology} and \cite{dullin1995complete} and providing a geometrical explanation for the difficulty of finding examples of this dynamical object. We will show that such sections never exist in many important cases. Some particular results include
\begin{theorem}\label{cotangent}
Let $(H,\omega,T^*M)$ be a Hamiltonian system on $T^*M$ equipped with the canonical symplectic form. Then no level energy surface of $H$ possesses a global transverse Poincar\'{e} section.\footnote{Note that such sections can (and do) still exist for magnetic symplectic forms on cotangent bundles.}
\end{theorem}
Such a statement eliminates the possibility of finding an ideal section for many classical dynamical systems including most systems coming from celestial mechanics. Having ruled out the possibility of global transverse Poincar\'{e} sections in what is, arguably, one of the most important cases, we give some sundry results on obstructions to their existence in other cases and examine the question of when an arbitrary hypersurface in a symplectic manifold possesses a cosymplectic structure. First we will give some preliminaries:
\begin{definition}
Let $H$ be a Hamiltonian function on a symplectic manifold $(\omega,M^{2n})$. A global surface of section for the Hamiltonian flow on closed energy hypersurface $Z=H^{-1}(c)$ is a compact, embedded hypersurface $\Sigma$ such that
\begin{itemize}
\item
the the boundary $\partial \Sigma$ consists of periodic trajectories (or invariant sets, for higher dimensions).
\item  $\Sigma \setminus \partial\Sigma$ is transverse to the flow.
\item Every orbit of the flow intersects $\Sigma$ in forward and backward time.
\end{itemize}
\end{definition}
The dynamics is then essentially governed by a return map of the $2n-2$ dimensional submanifold $\Sigma$. Note, however, that the existence of the boundary $\partial \Sigma$ can complicate things. In the ideal case, the Hamiltonian flow is everywhere transverse to $\Sigma$ and $\Sigma$ is then closed embedded submanifold. $\Sigma$ is then called a \emph{global transverse Poincar\'{e} section}
\begin{definition}
A cosymplectic manifold $Z$ of dimension $2n-1$ is a manifold equipped with a closed one-form $\alpha$ and closed two-form $\beta$ such that $\alpha \wedge \beta^{n-1}$ is a volume form on $Z$. 
\end{definition}

As $\alpha$ is closed, $ker(\alpha)$ gives a integrable distribution on $Z$. The associated foliation is regular. Moreover each leaf of the foliation $\mathcal{L}$ has an induced symplectic form given by $\iota_\mathcal{L}^*\beta$, where $\iota_\mathcal{L}$ is the canonical inclusion function. A cosymplectic manifold is a regular Poisson manifold equipped with a transverse Poisson vector field. In the case that $Z$ is a compact manifold and $\mathcal{L}$ is a compact leaf, upon dragging the leaf by the Poisson vector field $v$, one can see that in this case a cosymplectic manifold is equivalent to the following construction:
\begin{definition}
Let $\mathcal{L}$ be a symplectic manifold and $f$ a symplectomorphism of $\mathcal{L}$. A symplectic mapping torus with holonomy $f$ and fibre $\mathcal{L}$ is the suspension manifold
$$M_f =\frac{(I \times \mathcal{L})}{(0,x)\sim (1,f(x))}$$
\end{definition}
The equivalence of cosymplectic manifolds possessing a compact symplectic leaf and mapping tori was first proved by Liebermann in \cite{libermann1959automorphismes}. An exposition on their Poisson nature can be found in \cite{guillemin2011codimension}.

\paragraph{\textbf{Acknowledgements:}} The author is grateful to Holgar Dullin for useful conversations during the preparation of this material. The author is grateful to Robert Cardona and Amadeu Delshams for helpful comments on a first draft of this manuscript.

\section{Non-existence of sections for exact symplectic submanifolds}

In this section we will prove that Hamiltonian systems $(H,\omega,M)$ on exact symplectic submanifolds never possess global transverse Poincar\'{e} sections. Theorem \ref{cotangent} is then, of course, an important corollary of this fact.

\begin{proposition}\label{symplecticsubmanifold}
Let $(H,\omega,M)$ be a Hamiltonian system, $H^{-1}(c)=Z$ a compact level energy surface and $\Sigma \hookrightarrow Z$ a global transverse Poincar\'{e} section. Then $\Sigma$ is a symplectic submanifold of $(\omega, M)$
\end{proposition}
\begin{proof}
Let $i:\Sigma\hookrightarrow M$ be the inclusion map. We need to show that the restriction of $\omega$ to $\Sigma$, $i^*\Sigma$, is non-degenerate. Consider the tubular neighbourhood of $\Sigma$ in $Z$ formed by flowing $\Sigma$ by the vector field $X_H$ for a small time interval $(-\varepsilon,\varepsilon)$. Denote the time-$t$ flow of $X_H$ by $\Phi^t_{X_H}$ As $X_H$ is transverse to $\Sigma$, we have that $U$ is diffeomorphic to $\Sigma\times(-\varepsilon,\varepsilon)$ with an explicit diffeomorphism given by 
\begin{align*}
\Psi:\Sigma\times(-\varepsilon,\varepsilon)&\rightarrow U\\
\Psi(p,t)&=\Phi^t_{X_H}(p)
\end{align*}
We equip $U$ with the coordinates $(p,t)$. Now, let $v\in T_p\Sigma$ be such $i^*\omega(v,w)=0$ for all $w\in T_p\Sigma$. We will prove $v=0$. We have $\iota_{X_H}\omega(i_*v)=dH(i_*v)=0$ as $v$ is tangent to the level sets of $H$. Similarly, let $G$ be a function which is agrees with $\tilde{G}(p,t)=t$ on $U$ and let $X_G$ be the corresponding Hamiltonian vector field with Hamiltonian $G$. Then we have that $\iota_{X_G}\omega(i_*v)=dG(i_*v)=0$ as $v$ is tangent to the level sets of $G$. As $T_p\Sigma\oplus X_H \oplus X_G$ spans $T_pM$, $i_*v=0$ by the non-degeneracy of $\omega$. As $i$ is injective, then, $v=0$.
\end{proof}
It is well-known that \emph{exact} symplectic manifolds do not possess compact symplectic submanifolds as a symplectic form on a compact submanifold $\Sigma$ cannot be closed if $\dim\Sigma\geq 2$. Whence we have 
\begin{corollary}
Let $(H,\omega,M)$ be a Hamiltonian system on an exact symplectic manifold. Let $H^{-1}(c)=Z$ be compact. Then $Z$ does not possess a global transverse Poincar\'{e} section.
\end{corollary}
Recalling that the canonical symplectic form on a cotangent bundle $T^*M$ is exact then, we have theorem \ref{cotangent}.

\section{Topological Conditions on General Symplectic Submanifolds}
The previous proposition rules out global Poincar\'{e} sections in the case that the symplectic manifold is exact. In the case that it is non-exact, the following proposition will result in some topological obstacles to the existence of global Poincar\'{e} sections for both the energy level set and the ambient symplectic manifold.

\begin{proposition}
Let $Z$ be a compact energy level surface of a Hamiltonian system $(H,\omega,M)$ which possesses a global complete Poincar\'{e} section $\Sigma$. Then $Z$ is a cosymplectic manifold.
\end{proposition}

\begin{proof}
As noted in \cite{bolsinov1996topology}, the existence of a global Poincar\'{e} section equips $Z$ with the structure of mapping torus with fibre $\Sigma$ as follows: Let $p$ be a point in $\Sigma$. Denote by $T(p)$ the first return time of $p$ to the section sigma. $T(p)$ is automatically smooth and finite on $Z$ due to our assumption that the section is global transverse. The Poincar\'{e} map is given by 
\begin{align*}
\phi:\Sigma&\rightarrow\Sigma\\
\phi(p)&=\phi_{X_H}^{T(p)}(p)
\end{align*}
The following is then a diffeomorphism from a mapping torus $M_{\Sigma,\phi}$ with fibre $\Sigma$ and holonomy $\phi$ to $Z$.
\begin{align*}
\Psi:\frac{\Sigma\times[0,1]}{(0,p)\sim(1,\phi(p))} &\rightarrow Z \\
\Psi(p,t)&=\Phi_{X_H}^{tT(p)}(p)
\end{align*}
Such a mapping torus is a $\mathbb{S}^1$-fibre bundle and so comes equipped with a closed canonical nowhere vanishing one form $\alpha$, restricting to zero on fibres, which can be realised as the pull back of the generator $d\theta\in H^1(\mathbb{S}^1)$ under the projection map $\pi:M_{\Sigma,\phi}\rightarrow \mathbb{S}^1$ projecting the mapping torus to its base. Moreover, as $X_H$ is transverse to each fibre the restriction of $\omega$ to fibres is symplectic. Then $(\omega|_Z)^{n-1}\wedge\alpha$ is nowhere vanishing on $Z$, where $\omega|_Z$ is the restriction of $\omega$ to $Z$. Whence, $Z$ is cosymplectic.
\end{proof}

\begin{remark}
The fact that the Poincar\'{e} map is symplectic is of course part of classical literature and the above theorem could be proven by a easy extension of the method found in e.g. \cite{mcduff2017introduction}. However the above proof illuminates the cosymplectic nature of the manifold and is more suited to the following discussions.
\end{remark}

The above theorem puts immediate obvious restrictions on the topology of an energy surface possessing an global transverse Poincar\'{e} section using a well known fact on cosymplectic manifolds.

\begin{corollary}
Let $(M,\omega,H)$ be a Hamiltonian system and $Z=H^{-1}(c)$ a closed level energy set possessing a global transverse Poincar\'{e} section. Then $H^i(Z)$ is non-trivial for all $0\leq i\leq 2n-1$
\end{corollary}

\begin{proof}
Let $\alpha$ and $\beta$ be the defining one and two forms of the symplectic foliation. Then $\beta^i$ and $\alpha\wedge\beta^i$ are nowhere vanishing for $0\leq i \leq n-1$.
\end{proof}

Interestingly, the existence of a global transverse Poincar\'{e} section depends only on the level energy set of the Hamiltonian and not on the Hamiltonian itself.

\begin{theorem}
Let $(M,H,\omega)$ be a Hamiltonian system and $Z=H^{-1}(c)$ a closed submanifold of $M$ which is cosymplectic. Then $Z$ possesses a global transverse Poincar\'{e} section for the flow of $H$. 
\end{theorem}

\begin{proof}
By proposition \ref{symplecticvectorfield}, as $Z$ is cosymplectic there exists a symplectic vector field $X$ transverse to $Z$ such that the defining one form of the symplectic foliation on $Z$ is given by $\alpha=\iota_X(\omega)$. It is clear that $dH(X)\neq 0$ and so $\alpha(X_H)=\iota_X\omega(X_H)=-\omega(X_H,X)=dH(X)\neq 0$. This implies that $X_H$ is everywhere transverse to the symplectic foliation. Note that a leaf of the foliation for a general $\alpha$ may not be compact. However, as in the proof of the Tischler theorem \cite{tischler1970fibering} we can approximate $\alpha$ up to arbitrary precision by the rational sum of pullbacks of generators of the circle. Explicitly, for every $\varepsilon$, there exists some $\alpha'$ so that
$$\|\alpha-\alpha'\| < \varepsilon$$
and $\alpha'$ is of the form
$$\alpha'=\frac{1}{d} \Sigma_i n_i f_i^*d\theta$$
For $n_i\in \mathbb{N}, f_i:M\rightarrow \mathbb{S}^1$ where $d\theta$ is the canonical generator of $H^1(\mathbb{S}^1)$. Furthermore, as proved in \cite{tischler1970fibering}, the foliation defined by such a one-form is a fibration, each fibre of which is necessarily compact as $Z$ is.
As $\alpha(X_H) \neq 0$ is an open condition, then, we can find a one-form $\alpha'$ such that $\alpha'(X_H)\neq 0$ and $\alpha'$ is of the above form. As $M$ is compact and $X_H$ is everywhere transverse the symplectic foliation, an orbit of $X_H$ will intersect a chosen leaf $\mathcal{L}$ of the foliation in forward and backward time. Whence, $\mathcal{L}$ is a global transverse Poincar\'{e} section for $X_H$.
\end{proof}

\section{Examples with energy surfaces of all possible topologies}

We will now show that there are no further technical obstructions to the existence of an ideal section, to wit, given any cosymplectic manifold $N$ we form a Hamiltonian system $(H, M, \omega)$ which has as level energy surface $N$ and an associated global Poincar\'{e} section.

\begin{proposition}
Let $N$ be a cosymplectic manifold. There exists symplectic manifold $(M,\omega)$ and a Hamiltonian $H\in \mathcal{C}^\infty(M)$ such that $H^{-1}(0)=N$ and $X$ possesses an ideal section whose associated cosymplectic symplectic structure is that of $N$.
\end{proposition}

\begin{proof}
Consider the manifold $M=N\times\mathbb{S}^1$. Let $\alpha$ and $\beta$ denote the defining one and two forms of the cosymplectic structure on $N$ and $d\theta$ the canonical generator of $H^1(\mathbb{S}^1)$. Denote by $\pi_N$ and $\pi_{\mathbb{S}^1}$ the projection to the $N$ and $\mathbb{S}^1$ factors of $M$. Then $M$ possesses a symplectic form given by 
\begin{equation}
 \omega=\pi_N^*\beta+\pi_N^*\alpha\wedge\pi_{\mathbb{S}^1}^*d\theta  
\end{equation}
Now let $H$ be function which is constant on the submanifold $N\times\{0\}$, for example $H=\sin{\theta}$. Then the Hamiltonian vector field associated to $H$ is given by $c\partial_t$. Let $\mathcal{L}$ be a leaf of the symplectic mapping torus. Then $X_H$ is everywhere transverse to $\mathcal{L}$ and so $\mathcal{L}$ is an ideal section associated to the flow of $X_H$.
\end{proof}

\section{Cosymplectic Hypersurfaces}
Motivated by the previous sections, we ask which hypersurfaces of a symplectic manifold possess a cosymplectic structure, and so a global transverse Poincar\'{e} section. Paralleling the contact case, then, a global Poincar\'{e} section or cosymplectic structure exists on a level energy surface if and only if the hypersurface possesses a certain transverse vector field - in the case of the contact forms the associated vector field is \emph{Liouville}, in the cosymplectic case the vector field is \emph{symplectic}.
\begin{proposition}\label{symplecticvectorfield}
Let $(M,\omega)$ be a symplectic manifold and $Z$ a codimension $1$ hypersurface. Then $Z$ has an induced cosymplectic structure if and only if $Z$ possesses a transverse vector field which is symplectic at $Z$.
\end{proposition}

\begin{proof}
Assume that $Z$ possesses a transverse symplectic vector field. We will prove the existence of the defining one and two forms of a cosymplectic manifold. Let $X$ be a symplectic vector field everywhere transverse to the hypersurface $Z$. Then the form $\alpha=\iota_X \omega$ is nowhere vanishing due to the non-degeneracy of $\omega$. Furthermore $\alpha$ is closed as 
\begin{align*}
 \d(\iota_X\omega)&=\mathcal{L}_X\omega-\iota(d\omega)\\
                  &=0 
\end{align*}
and so $\alpha$ is the defining one-form of the foliation. Now let $i$ be the inclusion function $i:Z\hookrightarrow M$. Then $\iota^*(\omega)$ is closed, as $\omega$ is. Furthermore, $\alpha\wedge i^*(\omega)^{n-1}$ is a volume form: let $v_1$ be the symplectic vector field transverse to $Z$ and let $(v_2,...v_{2n})$ be a family of vectors spanning $T_pZ$. Identifying the vectors $v_i$ and $i_*v_i$ we have
\begin{align*}
    \alpha\wedge i^*\omega^{n-1}&(v_2,...v_{2n})\\
    &=\iota_{v_1}\omega\wedge \omega^{n-1}(v_2,...,v_{2n})\\
    &=\omega^{n-1}(v_1,...,v_{2n})\\
    \end{align*}
which is nowhere zero as $\omega$ is non-degenerate. $i^*\omega$ is then the defining two-form of the foliation.

Now let $(M,\omega)$ be a symplectic manifold with a cosymplectic hypersurface $Z$. Consider the symplectic form 
\begin{equation}
    \tilde{\omega}=\beta+\alpha\wedge dt
\end{equation}
where $dt$ is a one-form such that $dt(Y)=1$ for $Y$ a section of the (trivial) normal bundle of $Z$. Then the form $\omega$ is a symplectic form which agrees with the form $\omega$ on $Z$ and so $\tilde{\omega}$ is symplectomorphic to $\omega$ in an open neighbourhood of $Z$. But then $L_Y(\tilde{\omega})=d(i_Y\omega)=d\alpha=0$ and so $Y$ is a symplectic vector field transverse to $Z$.
\end{proof}

Finally we will prove that, in certain cases, the existence of a Hamiltonian system possessing an global transverse Poincar\'{e} section gives topological requirements on the ambient symplectic manifold, in addition to those on the hypersurface itself. First we need a simple lemma:
\begin{lemma}
Let $X$ be a vector field which transverse to a hypersurface $N$ and symplectic at $N$. If $H^1(M)=0$ then there is a vector field $X_{ext}$ such that $X_{ext}$ is a global symplectic vector field transverse to $N$.
\end{lemma}

\begin{proof}
Recall that as $M$ is simply connected all symplectic vector fields are Hamiltonian. Let $\alpha$ be the one form $\iota_X\omega$. As $N$ is closed $\alpha$ admits a global extension to $M$. Let $H$ be a primitive of $\alpha$. Then the Hamiltonian vector field associated to $H$ is a symplectic vector field whose restriction to $N$ is $X$.
\end{proof}

The idea for the following proposition was helpfully given by Klaus Niderkr\"{u}ger in \cite{326581}.

\begin{proposition}
Let $(M,\omega)$ be a compact simply connected symplectic manifold and $Z$ a connected hypersurface. Then there there exists no symplectic vector field transverse to $Z$.
\end{proposition}

\begin{proof}
As $M$ is symplectic it has an associated volume form given by $\omega^n$ and any symplectic vector field will preserve the symplectic volume of each component. As $M$ is simply connected, any closed hypersurface $Z$ separates $M$ into two components $G_0$ and $G_1$. But then the (either forward or backward) flow of the symplectic vector field transverse to such a $Z$ will map $G_0$ into a subset of $M$, properly containing $G_0$, of the same volume, which is a contradiction.
\end{proof}

\begin{corollary}
Let $(M,\omega)$ be a compact simply connected symplectic manifold and $H$ a Hamiltonian function on $M$. Then no level energy set of $H$ possesses a global transverse Poincar\'{e} section.
\end{corollary}

\bibliographystyle{plain}
\bibliography{references}
\end{document}